\documentclass{amsart}
\usepackage{amsthm,amssymb,amsmath,latexsym}
\usepackage{stmaryrd}		
\usepackage{hyperref,url}
\usepackage{graphicx}

\newtheorem{theorem}{Theorem}[section]
\newtheorem*{factortheorem*}{Theorem~\ref{factor}}
\newtheorem{proposition}[theorem]{Proposition}
\newtheorem{corollary}[theorem]{Corollary}

\theoremstyle{definition}
\newtheorem{definition}[theorem]{Definition}
\newtheorem{example}[theorem]{Example}

\newcommand{\colonequal}{\mathrel{\mathop:}=}

\newcommand{\xch}[2]{#1}
\newcommand{\reftext}[1]{#1}
\newcommand{\Left}[1]{\left}
\newcommand{\Right}[1]{\right}
\newcommand{\eqncr}{\\}
\newcommand{\noxml}[1]{}
\newcommand{\xmlstring}[1]{}

\begin{document}

\title[A characterization of $p$-automatic sequences]{A characterization of $p$-automatic sequences as columns of linear cellular automata}

\author{Eric Rowland}
\address{
	LaCIM \\
	Universit\'e du Qu\'ebec \`a Montr\'eal \\
	Montr\'eal, QC H2X 3Y7, Canada
}
\curraddr{
	D\'epartement de Math\'ematiques \\
	University of Li\`ege \\
	4000 Li\`ege, Belgium
}

\author{Reem Yassawi}
\thanks{The second author was partially supported by an NSERC grant.}
\address{Trent University, Peterborough, Canada}



\begin{abstract} 
We show that a sequence over a finite field $\mathbb F_q$ of
characteristic $p$ is $p$-automatic if and only if it occurs as a
column of the spacetime diagram, with eventually periodic initial
conditions, of a linear cellular automaton with memory over $\mathbb
F_q$. As a consequence, the subshift generated by a \xch{length-$p$
substitution}{length $p$-substitution} can be realized as a topological
factor of a linear cellular automaton.
\end{abstract}

\maketitle

\section{Introduction}\label{Introduction}

In a cellular automaton, each cell has a value at each time step, so it
is natural to consider the sequence of values taken by a given cell at
times $0$, $1$, $2$, etc. For a one-dimensional cellular automaton,
such a sequence of states is a column sequence in a two-dimensional
spacetime diagram of the cellular automaton. Some column sequences are
quite simple, such as the characteristic sequence of powers of~$2$,
which occurs as a column in the \xch{spacetime}{space-time} diagram of
\xch{rule}{Rule}~$90$ begun from a single 1 cell on background of $0$s;
this rule adds its two \xch{neighbors}{neighbours} modulo $2$. On the
other hand, some sequences are statistically random, such as the center
column of \xch{rule}{Rule}~$30$~\cite[\xch{p.}{page}~28]{NKS}.

A relatively unexplored question is the following. Given a sequence on
a finite alphabet, does this sequence occur as a column of a cellular
automaton spacetime diagram? Without additional restrictions, it is
possible to obtain any sequence (for example, the base-$10$ digits of
$\pi$) as a column: simply place the sequence in the initial condition,
and let the cellular automaton be the shift map $\sigma$. To avoid this
trivial case, we impose the restriction that initial conditions be
eventually periodic in both directions.

Wolfram found, by brute-force search, spacetime diagrams containing the
characteristic sequence of squares and the Thue--Morse
sequence~\cite[\xch{p.}{page}~1186]{NKS}. It is also possible to
construct spacetime diagrams with more exotic column sequences, such as
the characteristic sequence of primes~\cite[\xch{p.}{page}~640]{NKS}.

In this paper we study $p$-automatic sequences occurring as columns of
cellular automaton spacetime diagrams. We assume that $p$ is prime and
$\mathbb F_q$ is a finite field of characteristic $p$ throughout. Litow
and Dumas~\cite{ld} gave several examples of cellular automata
containing well-known $p$-automatic sequences as columns. They also
proved that each column of a linear cellular automaton over $\mathbb
F_q$, begun from an initial condition with finitely many nonzero
entries, is necessarily $p$-automatic. We use Litow and Dumas's
approach to establish the following characterization of $p$-automatic
sequences. (All relevant definitions are in Section~\ref{definitions}.)

\begin{theorem}\label{characterization}
A sequence of elements in $\mathbb F_q$ is $p$-automatic if and only if
it is a column of a spacetime diagram of a linear cellular automaton
with memory over $\mathbb F_{q}$ whose initial conditions are
eventually periodic in both directions.
\end{theorem}

Furthermore, the proof in each direction is constructive.
In particular, there is an algorithm to compute the local cellular
automaton rule, given a finite automaton for the $p$-automatic sequence.

In \cite{ahpps}, the authors study the automaticity of the
two-dimensional sequence of entries in the spacetime diagram generated
by a linear cellular automaton over the integers modulo~$m$, with
eventually constant initial conditions. Let such a cellular automaton
be generated by the polynomial $C(x) \in (\mathbb Z/(m\mathbb Z))[x]$.
They show that the two-dimensional spacetime diagram is $p$-automatic
if and only if the set of prime divisors of $m$ such that the $C(x)
\bmod p$ is not a monomial is either $\{p\}$ or the empty set. It
follows that a linear cellular automaton over $\mathbb F_p$ generates a
$p$-automatic spacetime diagram. Moreover, columns of this spacetime
diagram, being one-dimensional slices of a two-dimensional
$p$-automatic sequence, are also $p$-automatic.

As a consequence of \reftext{Theorem~\ref{characterization}} we are able to prove
the following.

\begin{theorem}\label{factor}
If $\mathbf{u}= (u_n)_{n \geq0}$ is $p$-automatic, then for some $q$
and $d\geq1$, there exists a linear cellular automaton
$\varPhi:({\mathbb F_{q}^{d}})^{\mathbb Z}\rightarrow({\mathbb
F_{q}^{d}})^{\mathbb Z}$ and a subsystem $(Y, \varPhi)$ of
$(({\mathbb F_{q}^{d}})^{\mathbb Z},\varPhi)$ such that
$(X_\mathbf{u}, \sigma)$ is a topological factor of $(Y,\varPhi)$.
\end{theorem}

We remark that for each $n \geq1$ a sequence is $p^{n}$-automatic if
and only if it is $p$-automatic~\cite[Theorem~6.6.4]{ash}, so
\reftext{Theorems~\ref{characterization} and \ref{factor}} also apply to
$p^{n}$-automatic sequences. Moreover, by injecting the alphabet of a
general $p$-automatic sequence into some $\mathbb F_q$, we can find an
image of that sequence, under the injection, as a column of a spacetime
diagram.

The proof of \reftext{Theorem~\ref{characterization}} appears in
Section~\ref{columns} along with some corollaries. In
Section~\ref{examples} we discuss an algorithm that, given a finite
automaton for a sequence, generates the desired cellular automaton with
memory, and we compute several examples. In Section~\ref{factors} we
prove \reftext{Theorem~\ref{factor}} and give conditions that ensure that the
factor mapping of \reftext{Theorem~\ref{factor}} is an embedding.

\section{Definitions and notation}\label{definitions}

In this section we recall definitions of some terms that we use.
Let $\varSigma_{k}=\{0,1, \ldots, k-1\}$.
If $n= \sum_{i=0}^{l} a_{i}k^{i}$ is the standard base-$k$
representation of $n$ with $0\leq a_{i}\leq k-1$ and $a_{l}\neq0$, define
$(n)_{k}$ to be the word $a_{0} a_{1}\cdots a_{l}$.
We start with the cumbersome, but necessary, definition of a finite
automaton that generates an automatic sequence:

\begin{definition}\label{dfao}
A \emph{deterministic finite automaton with output} (DFAO) is a 6-tuple
$(\mathcal S, \varSigma_{k},\delta, s_{0}, \mathcal A, \omega)$, where
$\mathcal S$ is a finite set of ``states'', $s_{0}\in\mathcal S$ is the
\emph{initial state}, $\mathcal A$ is a finite alphabet,
$\omega:\mathcal S\rightarrow\mathcal A$ is the \emph{output function},
and $\delta:\mathcal S\times\varSigma_{k}\rightarrow \mathcal S$ is the
\emph{transition function}.
\end{definition}

In the symbolic dynamics literature, $\omega$ is also known as a
\emph{coding} or a \emph{letter-to-letter projection}. The function
$\delta$ extends in a natural way to the domain $\mathcal
S\times\varSigma_{k}^{*}$. Namely, define $\delta(s, a_0 a_1 \cdots a_l)
:=\delta(\delta(s, a_0), a_1 \cdots a_l)$ recursively. This
allows us to feed the automaton with base-$k$ representations of
natural numbers:

\begin{definition}
A sequence $(u_n)_{n \geq0}$ of elements in $\mathcal A$ is
\emph{$k$-automatic} if there is a DFAO $(\mathcal
S,\varSigma_{k},\delta, s_{0}, \mathcal A, \omega)$ such that $u_{n} =
\omega(\delta(s_{0},(n)_{k}))$ for all $n \geq0$.
\end{definition}

\begin{example}
The Thue--Morse sequence is the sequence $(u_n)_{n \geq0} = 0, 1, 1,
0, 1, 0, 0, 1, \dots$ where $u_n = 0$ if the number of occurrences of
$1$ in the binary representation of $n$ is even and $u_n = 1$ otherwise.
The Thue--Morse sequence is $2$-automatic, and it is generated by the
following automaton, where the two states are labeled with their images
under~$\omega$.
\begin{center}
	\includegraphics[scale=.7]{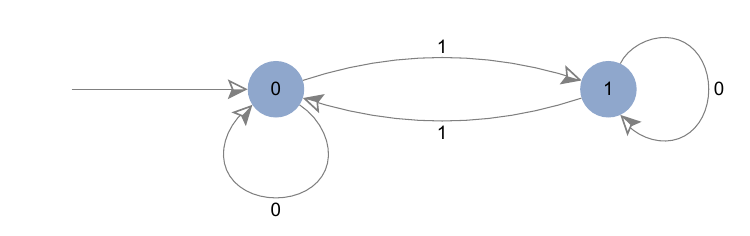}
\end{center}
\end{example}

Let ${\mathbb N} :=\{0,1, \ldots\}$. If ${\mathbb M} = {\mathbb Z}$ or
${\mathbb M} = {\mathbb N}$, then the space of all ${\mathbb
M}$-indexed sequences from ${\mathcal A}$ is written as ${\mathcal
A}^{\mathbb M}$, and an element in ${\mathcal A}^{\mathbb M}$, a
\emph{configuration}, is written $R=(R(m))_{m \in \mathbb M}$. (Our
configurations will be thought of as rows of a two-dimensional array.)
Let ${\mathcal A}$ be endowed with the discrete topology and ${\mathcal
A}^{\mathbb M}$ with the product topology; then ${\mathcal A}^{\mathbb
M}$ is a Cantor space, that is, a zero-dimensional perfect compact
metric space. If $b\in{\mathcal A}$ and $m \in\mathbb M$, the
clopen sets $\{R : R(m)=b\}$ generate a countable basis for the
topology on ${\mathcal A}^{\mathbb M}$. The (left) \emph{shift map}
$\sigma: {\mathcal A}^{\mathbb M} \rightarrow {\mathcal A}^{\mathbb M}$
is the map defined as $(\sigma(R))(m) =R(m+1)$. Given two configuration
spaces ${\mathcal A}^{\mathbb M}$ and ${\mathcal B}^{\mathbb M}$, we
shall use $\sigma$ to refer to the shift map on either of these spaces.
Recall that if $\varPhi:{\mathcal A}^{\mathbb M} \rightarrow{\mathcal
B}^{\mathbb M}$, we say that $\varPhi$ \emph{commutes with the shift}
if $\sigma\circ\varPhi= \varPhi\circ\sigma$.

\begin{definition}
A (one-dimensional) \emph{cellular automaton with memory $d$} is a
continuous, $\sigma$-commuting map $\varPhi: ({\mathcal
A^{d}})^{\mathbb Z}\rightarrow{\mathcal A}^{\mathbb Z}$.
\end{definition}

By memory here we mean ``time'' memory, and this will become clear
after \reftext{Definition~\ref{spacetime}}. To recover the classical
definition of a cellular automaton, we let $d=1$. In this case the
cellular automaton only needs to know the current configuration, and
nothing of the configuration's past. The Curtis--Hedlund--Lyndon
theorem \cite{h} states that $\varPhi$ is a cellular automaton if there
is some \emph{local rule} $\phi:\mathcal{A}^{l+r+1 }\rightarrow
\mathcal{A}$ for some $l\geq0$ (the \emph{left radius of $\phi$}) and
$r\geq0$ (the \emph{right radius of $\phi$)}, such that for all $R
\in{\mathcal A}^\mathbb{Z}$ and all $m \in\mathbb{Z}$,
%
\begin{eqnarray}
\label{ca_local} \Left{bigl}(\varPhi(R)\Right{bigr}) (m) = \phi\Left{bigl}(R(m-l),
R(m-l+1), \ldots, R(m+r)\Right{bigr}) .
\end{eqnarray}
Conversely, any local rule $\phi$ defines a cellular automaton $\varPhi$
using \xch{identity}{Identity} \reftext{(\ref{ca_local})}.

The Curtis--Hedlund--Lyndon theorem also holds for a cellular automaton
with memory, so that there is a local rule $\phi:(\mathcal
{A}^d)^{l+r+1}\rightarrow\mathcal{A}$ satisfying
\reftext{(\ref{ca_local})}, and conversely any such local rule defines
a cellular automaton with memory. We shall often use the fact that the
domain of a cellular automaton with memory $d$ is also $({\mathcal
A^{\mathbb Z}})^{d}$.

\begin{definition}\label{spacetime}
If $\varPhi:(\mathcal{A}^d)^\mathbb{Z}\rightarrow\mathcal{A}^\mathbb{Z}$
is a cellular automaton with memory~$d$, then a \emph{spacetime
diagram} for $\varPhi$ with initial conditions $R_{0}, \ldots, R_{d-1}$
is the sequence $(R_n)_{n \geq0}$ where we inductively define $R_{n}
:=\varPhi(R_{n-d}, \ldots, R_{n-1})$ for $n\geq d$.
\end{definition}

We visualize a spacetime diagram by letting time evolve down the page:
the $n$th row $R_{n}$ represents the configuration at time $n$, and
$R_{n}(m)$, the entry on \xch{row}{Row} $n$ and \xch{column}{Column}
$m$ of the spacetime diagram, is the state of the $m$th cell at
time~$n$. Whereas each row in an ordinary cellular automaton (with
memory $1$) is determined by the previous row, in a cellular automaton
with memory $d$ each row is determined by the previous $d$ rows. To be
brief, we will often use the term ``cellular automaton'' to mean ``a
spacetime diagram of the cellular automaton''.

Now suppose that $\mathcal A$ is the finite field $\mathbb F_{q}$. In
this case $\mathbb F_{q}^{\mathbb Z}$ and $(\mathbb F^d_{q})^{\mathbb
Z}$ are groups, with componentwise addition; they are also $\mathbb
F_{q}$-vector spaces.

\begin{definition} \label{linear}
We say that the cellular automaton $\varPhi:(\mathbb
F_{q}^{d})^{\mathbb Z}\rightarrow \mathbb F_{q}^{\mathbb Z}$ with
memory $d$ is \emph{linear} if $\varPhi$ is an $\mathbb F_{q}$-linear
map.
\end{definition}

Thus the Curtis--Hedlund--Lyndon theorem implies that the memory-$d$
cellular automaton $\varPhi$ is linear if and only if there exist
coefficients $f_{j,i}\in\mathbb F_{q}$ for $-l \leq j\leq r$ and $0\leq
i \leq d-1$ such that $(\varPhi(R_{0}, \ldots,R_{d-1}))(m) =
\sum_{i=0}^{d-1}\sum_{j=-l}^{r} f_{j,i} R_{i}(m+j) $ for all $R_{0},
\ldots, R_{d-1} \in\mathbb F_q^\mathbb Z$ and $m \in \mathbb Z$. An
example of a linear cellular automaton with memory $1$ is
\xch{rule}{Rule}~$90$, whose field is $\mathbb F_2$ and whose local
rule is $\phi(a,b,c) = a+c$. Begun from the initial condition $R_0$
where $R(0) = 1$ and $R(m) = 0$ for all $m \neq0$,
\xch{rule}{Rule}~$90$ computes the array of binomial coefficients
modulo~$2$. In fact, Pascal's triangle modulo $p$ is the spacetime
diagram of a linear cellular automaton with the corresponding initial
condition. These rules have been studied extensively in the literature,
as have been linear cellular automata in general; the algebraic
properties of the local rule yield much theoretic structure.

\begin{example}
\reftext{Fig.~\ref{Thue--Morse}} shows the first $256$ rows of the
spacetime diagram of a linear cellular automaton with memory $12$ over
$\mathbb F_2$, where $0$ is rendered as a white cell and $1$ is
rendered as a black cell. The ``center'' column, containing the top
vertex of the triangular region, consists of the Thue--Morse sequence.
This column is highlighted by rendering $0$ as red. It lies to the left
of a column which is identically zero. We compute the local rule for
this cellular automaton in Section~\ref{examples}.
\end{example}

\begin{figure}
	\begin{center}
		\scalebox{.8}{\includegraphics{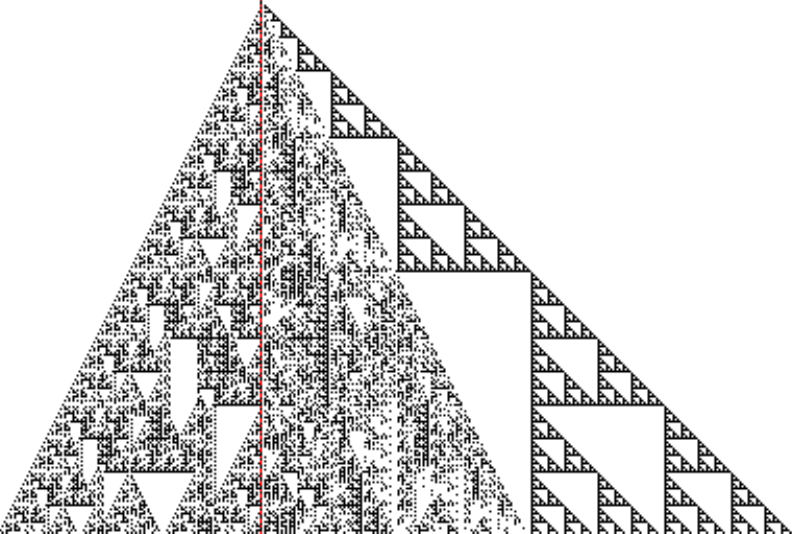}}
		\caption{Spacetime diagram of a linear cellular automaton with memory $12$ containing the Thue--Morse sequence as a column.}
		\label{Thue--Morse}
	\end{center}
\end{figure}

\subsection{Classical results}

We recall some results that we shall use. As before, we use $p$ to
denote a prime and $q$ to denote a power of~$p$. Let $\mathbb F_{q}$ be
the field of cardinality~$q$. If $(u_n)_{n \geq0}$ is $p$-automatic
generated using the DFAO $(\mathcal S, \varSigma_{p},\delta, s_{0},
\mathcal A, \omega)$, find $q$ such that $|\mathcal A|\leq q$. By
injecting $\mathcal A$ into $\mathbb F_{q}$, we can, and henceforth do,
assume that the output function $\omega$ has range in~$\mathbb F_q$.

Recall that $\mathbb F_{q}[t]$, $\mathbb F_{q}(t)$, and $\mathbb
F_{q}((t))$ are the sets of polynomials, rational functions, and formal
Laurent series respectively with coefficients in $\mathbb F_{q}$.
Elements of $\mathbb F_{q}((t))$ are expressions of the form
$F(t)=\sum_{n\geq n_{0}}u_{n}t^{n}$, where $u_{n}\in\mathbb F_{q}$ and
$n_{0}\in\mathbb Z$. We can also define polynomials, rational functions
and formal Laurent series in several variables. For~us, a formal
Laurent series in $t, x$ is an element of $\mathbb F_{q}((x))((t))$.
The Laurent series $F(t)$ is \emph{algebraic over $\mathbb F_{q}(t)$}
if there exists a nonzero polynomial $P(t, x) \in\mathbb F_{q}[t, x]$
such that $P(t, F(t)) = 0$. Finally we define the \emph{$q$-kernel} of
a sequence $(u_{n})_{n\geq 0}$ to be the collection of sequences
$\{(u_{q^{k}n+r})_{n\geq0}: k\geq0, \ 0\leq r\leq q^{k}-1\}$. The
equivalence of \xch{statements}{Statements}~1 and 2 in
\reftext{Theorem~\ref{Christol}} is known as Christol's theorem;
see~\cite{Christol,ckmr}. The equivalence of
\xch{statements}{Statements}~2 and 3 dates back to Eilenberg~\cite{ei}.

\begin{theorem}\label{Christol}
Let $(u_n)_{n \geq0}$ be a sequence of elements in $\mathbb F_q$. The
following are equivalent:
\begin{enumerate}
\item $F(t)=\sum_{n \geq0}u_{n}t^{n}$ is algebraic over
$\mathbb F_{q}(t)$.
\item The $q$-kernel of $(u_{n})_{n\geq0}$ is finite.
\item $(u_{n})_{n \geq0}$ is $q$-automatic.
\end{enumerate}
\end{theorem}

Let  $E(t,x) = \sum_{n \geq n_{0}}
\sum_{m \in\mathbb Z} a_{n,m}t^{n}x^{m}$ be a formal Laurent series in two variables $t$ and~$x$.
In this case the \emph{diagonal} is the formal Laurent series
\begin{eqnarray*}
\sum_{n \geq n_{0}}a_{n,n}t^{n}
\end{eqnarray*}
in one variable.
Similarly we define the $m$th \emph{column} of $E(t,x)$ to be
\begin{eqnarray*}
\sum_{n \geq n_0} a_{n,m} t^{n}.
\end{eqnarray*}

We say that $E(t,x)$ is a \emph{rational} series if there exist
polynomials $Q(t,x)$ and $P(t,x)$ such that $P(t,x) E(t,x) = Q(t,x)$.
The following result is due to Furstenberg~\cite{Fur}.

\begin{theorem}\label{Furstenberg}
For a Laurent series $F(t) \in\mathbb F_q((t))$ to be algebraic over
$\mathbb F _q(t)$, it is necessary and sufficient that it is the
diagonal of a rational Laurent series $E(t,x) \in F_{q}((x))((t))$.
\end{theorem}

The necessary direction of Furstenberg's theorem follows from the
following two propositions in \cite{Fur}, which we will use in
Section~\ref{columns}.

\begin{proposition}\label{prop_1}
Suppose that the Laurent series $F(t) = \sum_{n \geq n_0} c_n t^n \in
\mathbb F_{q}((t))$ is algebraic over $\mathbb F_q(t)$. Then there
exists $r^* \geq n_0$, $m \geq0$, and a polynomial
\begin{eqnarray*}
P^*(t,x) = A^*_{0}(t)x+A^*_{1}(t)x^{p}+
\cdots+A^*_{m}(t)x^{p^m}+B^*(t),
\end{eqnarray*}
with $A^*_{i}(t), B^*(t) \in\mathbb F_q[t]$ and $A^*_{0}(t)$ not
divisible by~$t$, such that $F(t) = R^*(t) + t^{r^*} G^*(t)$, $R^*(t) =
\sum_{n=n_0}^{r^*-1} c_n t^n$, and $P^*(t,G^*(t))=0$.
\end{proposition}

\reftext{Proposition~\ref{prop_1}} is similar to a result known as Ore's
lemma~\cite[Lemma~12.2.3]{ash}.

Let $P^{(0,1)}$ denote the derivative of a function $P$ with respect to
its second argument.

\begin{proposition}\label{prop_2}
Suppose that the series $G(t) = \sum_{n \geq1} c_n t^n \in\mathbb
F_q((t))$ with $G(0) = 0$ satisfies $P(t,G(t))=0$, where $P(t,x)
\in\mathbb F_{q}[t,x]$ and $P^{(0,1)} (0,0)\neq0$. Then $G(t)$ is the
diagonal of the unique series expansion of
\begin{eqnarray*}
\frac{x^{2}P^{(0,1)}(tx,x)}{P(tx,x)}.
\end{eqnarray*}
\end{proposition}

\section{Columns of linear cellular automata}\label{columns}

In this section we prove \reftext{Theorem~\ref{characterization}}.
\reftext{Theorems~\ref{Christol} and \ref{Furstenberg}} are main
ingredients. To work with spacetime diagrams algebraically, we
represent the spacetime diagram of a cellular automaton as a bivariate
series, a technique that, to our knowledge, was first used by Martin,
Odlyzko, and Wolfram \cite{Martin--Odlyzko--Wolfram}. If $a_{n,m}$ is
the entry of the spacetime diagram on row $n \in \mathbb N$ and column
$m \in\mathbb Z$, then the series $E(t,x) = \sum_{n \geq 0} \sum_{m
\in\mathbb Z} a_{n,m} t^n x^m$ encodes the entire cellular automaton
evolution from the initial condition $R_{0}$. We identify the $n$th row
$R_n$ of the spacetime diagram with the series $R_n(x) = \sum_{m
\in\mathbb Z} a_{n,m} x^m$, which is the coefficient of $t^n$ in
$E(t,x) = \sum_{n \geq0} R_n(x) t^n$.

We begin with the easier direction of
\reftext{Theorem~\ref{characterization}}. Litow and Dumas~\cite{ld}
established the case where the initial condition is zero everywhere
save the central entry and the memory is $1$. The proof of the more
general statement is similar.

\begin{theorem}\label{automatic_sequence_construction}
Let $\varPhi$ be a linear cellular automaton rule with memory $d$ over
$\mathbb F_q$. Let $R_0, \dots, R_{d-1} \in\mathbb F_q^\mathbb Z$ be
rows that are eventually periodic in both directions. For each $m
\in\mathbb Z$, the column sequence $(R_n(m))_{n \geq0}$ is
$p$-automatic.
\end{theorem}

\begin{proof}
The linearity of $\varPhi$ is equivalent to the existence of Laurent
polynomials $C_1(x), \dots, C_d(x)$ such that for all $n \geq d$
\begin{eqnarray*}
R_n(x) = \sum_{i=1}^d
C_i(x) R_{n-i}(x).
\end{eqnarray*}
Let $C_0(x) :=-1$; then $\sum_{i=0}^d C_i(x) R_{n-i}(x) = 0$
for all $n \geq d$, so
\begin{eqnarray*}
\Left{Biggl}(\,\sum_{i=0}^d
C_i(x) t^i \Right{Biggr}) E(t,x) &=& \Left{Biggl}(\,\sum
_{i=0}^d C_i(x) t^i \Right{Biggr}) \Left{biggl}(\,\sum_{j \geq0}
R_j(x) t^j \Right{biggr})
\\
&=& \sum_{n \geq0} \Left{Biggl}(\,\sum_{i=0}^d C_i(x) R_{n-i}(x)
\Right{Biggr}) t^n
\\
&=& \sum_{n=0}^{d-1} \Left{Biggl}(\,\sum
_{i=0}^d C_i(x)R_{n-i}(x) \Right{Biggr}) t^n.
\end{eqnarray*}
Each $R_{n-i}(x)$ is a rational expression since it is the sum of two
one-sided eventually periodic series. Therefore $ (\sum_{i=0}^d C_i(x)
t^i ) E(t,x)$ is a rational expression in $t$ and~$x$, and since
$\sum_{i=0}^d C_i(x) t^i$ is also rational this implies that $E(t,x)$
is also rational. Column $m$ of $E(t,x)$ is the diagonal of $x^{-m} E(t
x, x)$ and therefore by \reftext{Theorem~\ref{Furstenberg}} is
algebraic over~$\mathbb F_q$. Hence, by
\reftext{Theorem~\ref{Christol}}, the sequence of entries in column $m$
is $p$-automatic.
\end{proof}

For the other direction of \reftext{Theorem~\ref{characterization}}, we
basically reverse the steps of the previous proof. The difficulty
arises in obtaining a recurrence $\sum_{i=0}^d C_i(x) R_{n-i}(x) = 0$
in which $C_0(x)$ is a (nonzero) monomial. It is necessary that
$C_0(x)$ is a monomial so that each $\frac {C_i(x)}{C_0(x)}$ is a
Laurent polynomial, and hence the update rule that determines the value
of each cell is local. Since the coefficients $C_i(x)$ come from the
denominator of $E(t,x)$ and the denominator of $E(t,x)$ comes (by
\reftext{Theorem~\ref{Furstenberg}}) from a polynomial equation
satisfied by $F(t)$, we seek a polynomial $P(t,x)$ where the
coefficient of $t^0$ is a nonzero monomial in $x$, and where $P(t,
F(t)) = 0$. The following proposition allows us to find such a
polynomial.

\begin{proposition}\label{technical_1}
Suppose that $F(t) = \sum_{n\geq0} u_{n}t^{n} \in\mathbb F_{q}\llbracket t\rrbracket$ is
algebraic over $\mathbb F_{q}(t)$. Then there exist $G(t)\in\mathbb
F_{q}\llbracket t\rrbracket$ and $P(t, x)\in\mathbb F_{q}[t,x]$ of the form
\begin{eqnarray*}
P(t,x) = A_{0}(t) x +A_{1}(t) x^{p} +
\cdots+A_m(t)x^{p^m}+B(t)
\end{eqnarray*}
with $A_{i}(t), B(t) \in\mathbb F_{q}[t]$ for $0 \leq i \leq m$ such that
\begin{enumerate}
\item $F(t) = R(t) + t^r G(t)$ for some $r \geq0$ and $R(t) \in
\mathbb F_{q}[t]$,
\item $G(0)=0$,
\item $A_{0}(0) \neq0$,
\item $B(0)=A_{i}(0) = 0$ for $1\leq i \leq m$, and
\item $P(t, G(t)) = 0$.
\end{enumerate}
\end{proposition}

\begin{proof}
\reftext{Proposition~\ref{prop_1}} tells us that for some $r^{*}\geq0$ there
exist polynomials $A_i^*(t), B^*(t) \in\mathbb F_{q}[t]$ such that
$A_{0}^{*}(0)\neq0$ and $G^{*}(t) = \sum_{n\geq0} u_{n+r^{*}}t^{n}$
satisfies $P^{*}(t,G^{*}(t)) =0$, where
\begin{eqnarray*}
P^{*}(t, x) = A_{0}^{*}(t) x +
A_{1}^{*}(t) x^{p} + \cdots+
A_m^{*}(t) x^{p^m} + B^{*}(t).
\end{eqnarray*}
Let $r = r^* + 1$ and $R(t) = u_0 + u_1 t + \cdots+ u_{r} t^{r}$. Write
$G^*(t) = u_{r - 1} + u_r t + t G(t)$ so that $G(t) = \sum_{n \geq1}
u_{n + r} t^n$. Expanding $P^*(t, u_{r - 1} + u_r t + t x)$ and using
that $(a+b)^p = a^p + b^p$ in characteristic $p$ shows that $x = G(t)$
satisfies
\begin{eqnarray*}
A_{0}^{*}(t) t x + A_{1}^{*}(t) t^p x^p + \cdots+ A_m^{*}(t) t^{p^m}
x^{p^m} + B^{**}(t) = 0
\end{eqnarray*}
for some polynomial $B^{**}(t)$. Since each formal power series
\[
A_i^*(t) t^{p^i} G(t)^{p^i} = A_i^*(t) t^{2 p^i} \Left{bigl}(G(t)/t\Right{bigr})^{p^i}
\]
 is divisible by $t^2$, $B^{**}(t)$ is also divisible by $t^2$.
Dividing by~$t$, let $A_i(t) = A_i^*(t) t^{p^i - 1}$, $B(t) =
B^{**}(t)/t$, and
\begin{eqnarray*}
P(t, x) = A_0(t) x + A_1(t) x^p + \cdots+
A_m(t) x^{p^m} + B(t).
\end{eqnarray*}
One verifies that the conclusions of the proposition are satisfied.
\end{proof}

We may now prove the other direction of \reftext{Theorem~\ref{characterization}}.

\begin{theorem}\label{cellular_automaton_construction}
If $(u_n)_{n \geq0}$ is a $p$-automatic sequence of elements in
$\mathbb F _q$, then $(u_n)_{n\geq0}$ occurs as a column of a linear
cellular automaton with memory over $\mathbb F_{q}$, whose initial
condition rows have finitely many nonzero entries.
\end{theorem}

\begin{proof}
By \reftext{Proposition~\ref{technical_1}}, there exist $G(t)$ and
$P(t,x)$ with $A_0(t)$ and $r$ as described. Using the notation of that
proposition, we shall show that the shifted sequence $u_{r+1},
u_{r+2},\dots$ can be found as a column of a spacetime diagram of a
linear cellular automaton with memory, and then we describe how to
reinstate the initial terms. Let us write $P(t,x) =
\sum_{i=0}^{d}C_{i}(x)t^{i}$ where $C_{i}(x)\in \mathbb F_{q}[x]$.
Conclusion~4 of \reftext{Proposition~\ref{technical_1}} implies that
$C_{0}(x) = A_{0}(0)x$, and \xch{conclusion}{Conclusion}~3 is that
$A_0(0) \neq0$.

Conclusions~2, 5, and 3 of \reftext{Proposition~\ref{technical_1}}
imply that the conditions of \reftext{Proposition~\ref{prop_2}} are
met. Therefore $G(t)$ is the diagonal of $\frac{x^{2}
P^{(0,1)}(tx,x)}{P(tx,x)}$. It follows that $G(t)$ is
\xch{column}{Column} $-2$ of $\frac{P^{(0,1)}(t,x)}{P(t,x)}$. Since the
coefficient of $t^0$ in $P(t,x)$ is the monomial $A_0(0) x$, write
$P(t,x) = A_0(0) x + t Q(t,x)$ where $Q(t,x) \in\mathbb F_q[t,x]$. Then
expand
\begin{eqnarray*}
\frac{P^{(0,1)}(t,x)}{P(t,x)} &=& \frac{P^{(0,1)}(t,x)}{A_0(0) x}
\cdot\frac{1}{1 + \frac{t Q(t,x)}{A_0(0) x}} =
\frac{P^{(0,1)}(t,x)}{A_0(0) x} \sum_{n \geq0} \Left{biggl}(-\frac
{Q(t,x)}{A_0(0) x} \Right{biggr})^n t^n \\
&=& \sum_{n \geq0} R_n(x) t^n
\end{eqnarray*}
as a series in $t$. Each $R_{n}(x)$ is a Laurent polynomial.

Create a two-dimensional array where the entry in \xch{row}{Row} $n
\in\mathbb N$ and \xch{column}{Column} $m \in\mathbb Z$ is the
coefficient of $x^m$ in $R_n(x)$. By
\reftext{Proposition~\ref{prop_2}}, \xch{column}{Column} $-2$ of this
array consists of the sequence $0, u_{r+1}, u_{r+2}, \dots$. It remains
to show that the array is the spacetime diagram of a cellular automaton
with memory and to then restore the terms $u_0, u_1, \dots, u_{r}$.

Since the series $\sum_{n \geq0} R_n(x) t^n$ is rational, the sequence
$(R_n(x))_{n \geq0}$ satisfies a linear recurrence with coefficients
$C_0(x), \dots, C_d(x)$. Namely, multiplying both sides by $P(t,x)$
gives
\begin{eqnarray*}
P^{(0,1)}(t,x) &=& \sum_{i=0}^{d}C_{i}(x)t^{i} \sum_{j \geq0} R_j(x)
t^{j} = \sum_{n \geq0} \Left{biggl}(\,\sum_{i+j=n} C_i (x)R_j(x)
\Right{biggr}) t^n
\\
&=& \sum_{n=0}^d \Left{Biggl}(\,\sum_{i=0}^n C_i(x)
R_{n-i}(x) \Right{Biggr}) t^n + \sum
_{n \geq d+1} \Left{Biggl}(\,\sum_{i=0}^d
C_i(x) R_{n-i}(x) \Right{Biggr}) t^n,
\end{eqnarray*}
and since $P^{(0,1)}(t,x)$ is a polynomial with $\deg_t P^{(0,1)}(t,x)
\leq d$, we have
\[
\sum_{i=0}^d C_i(x) R_{n-i}(x) = 0
\]
for all $n \geq d + 1$. Solving for $R_n(x)$ gives
\begin{eqnarray*}
R_n (x)= -\sum_{i=1}^d
\frac{C_{i}(x)}{C_{0}(x)} R_{n-i}(x)
\end{eqnarray*}
for all $n \geq d + 1$, where each $\frac{C_{i}(x)}{C_{0}(x)} $ is a
Laurent polynomial in~$x$. Therefore the coefficient of $x^m$ in
$R_n(x)$ depends only on the coefficients of $x^{m + 1 - \max_{i} \deg
C_i(x)}, \dots, x^{m+1}$ in $R_{n-1}(x), \dots , R_{n-d}(x)$. In
particular, each entry of the two-dimensional array is computed by the
same local rule. Therefore the coefficients of $R_1(x), R_2(x), \dots$
form the rows of a spacetime diagram of a cellular automaton with
memory $d$, where the first $d$ rows consist of initial conditions and
rows $R_{d+1}, R_{d+2}, \dots$ are computed by the local rule. Remove
row $R_0$, since no other rows depend on it and the coefficient of
$x^{-2}$ in $R_0(x)$ is $0$ rather than $u_r$. Finally, we restore the
initial $r+1$ terms of $(u_n)_{n \geq0}$. We do this by redefining
$R_0(x) :=u_r x^{-2}$ and defining $R_{-r}(x) :=u_0 x^{-2}, R_{1-r}(x)
:=u_1 x^{-2}, \dots, R_{-1}(x) :=u_{r-1} x^{-2}$. We  trivially
increase the memory of the local rule from $d$ to $d + r + 1$ without
actually introducing dependence of $R_n$ on rows $R_{n-(d+1)}, \dots,
R_{n-(d+r+1)}$. Then the local cellular automaton rule with memory $d +
r + 1$, run from initial conditions $R_{-r}, \dots, R_d$, produces a
spacetime diagram where the sequence $(u_n)_{n \geq0}$ occurs in
\xch{column}{Column}~$-2$.
\end{proof}

The construction in the proof of
\reftext{Theorem~\ref{cellular_automaton_construction}} gives us some
additional information about the spacetime diagram. For example, since
$C_0(x) = A_0(0) x$ has degree $1$, the right radius of the local rule
is at most~$1$. Therefore the left boundary of the nonzero triangular
region in the spacetime diagram grows with speed at most $1$ cell per
step.

Additionally, \xch{column}{Column} $-1$ is identically $0$. This column
can be seen, for example, immediately to the right of the Thue--Morse
column in \reftext{Fig.~\ref{Thue--Morse}} (and helps the reader
identify the location of the desired sequence in the rest of our
diagrams). To see that this is the case, factor $P(t,x) = (x - G(t))
Q(t,x)$ for some $Q(t,x) \in\mathbb F_q((t))[x]$, following the proof
of \reftext{Proposition~\ref{prop_2}} in \cite{Fur}. Then we have
\begin{eqnarray*}
\frac{P^{(0,1)}(t,x)}{P(t,x)} = \frac{1}{x - G(t)} + \frac
{Q^{(0,1)}(t,x)}{Q(t,x)}.
\end{eqnarray*}
These two summands contain the entries of the two halves of the
spacetime diagram. Since $Q(0,0) \neq0$, the exponent of $x$ in each
nonzero term of the series $\frac{Q^{(0,1)}(t,x)}{Q(t,x)}$ is
nonnegative. Moreover, the only nonzero term in the series $\frac{1}{x
- G(t)} = \frac{1}{x} \sum_{n \geq0} (\frac{1}{x} G(t))^n$ whose
exponent of $x$ is greater than $-2$ is $\frac{1}{x}$, which appears in
$R_0(x)$, which we removed in the proof of
\reftext{Theorem~\ref{cellular_automaton_construction}}.

In Section~\ref{examples} we discuss an algorithm to generate the
polynomial $P(t,x)$ of \reftext{Proposition~\ref{technical_1}} and thus
the cellular automaton as constructed in
\reftext{Theorem~\ref{cellular_automaton_construction}}. We also work
through some examples. First though we mention a few corollaries.

\begin{corollary}\label{without_memory}
If $(u_n)_{n \geq0}$ is a $p$-automatic sequence, then $(u_n)_{n\geq
0}$ is the letter-to-letter projection of a sequence $(v_n)_{n\geq0}$ which
occurs as a column of a linear cellular automaton (without memory)
whose initial condition is eventually periodic in both directions.
\end{corollary}

\begin{proof}
\reftext{Theorem~\ref{cellular_automaton_construction}} guarantees the
existence of a linear rule $\varPhi$ with memory $d+r+1$ such that
$(u_{n})_{n\geq 0}$ occurs as a column of some spacetime diagram of
$\varPhi$. Wrap every $(d+r+1)$-tuple of consecutive cells in each
column into a single cell. In other words, consider the new alphabet
$\mathbb F _{q}^{d+r+1}$. The new cellular automaton $\varPhi^{*}:
(\mathbb F _{q}^{d+r+1})^{\mathbb Z}\rightarrow (\mathbb
F_{q}^{d+r+1})^{\mathbb Z}$, without memory, has the same left and
right radius as the old, and has a local rule which takes the
``central'' cell, discards the top entry of that $(d+r+1)$-tuple,
shifts the most recent $d+r$ entries up, and inserts the output of the
old local rule at the bottom. This construction means that there is
some sequence $(v_{n})_{n\geq0}\in(\mathbb F_{q}^{d+r+1})^{\mathbb N}$
which occurs as a column of some spacetime diagram for $\varPhi^{*}$,
and such that if we project each $v_{n}$ to its first entry, we obtain
$(u_{n})_{n\geq0}$. It is straightforward that $\varPhi^*$ is linear.
\end{proof}

\begin{corollary}\label{without_projection}
If $(u_n)_{n \geq0}$ is a $p$-automatic sequence, then $(u_n)_{n\geq
0}$ occurs as a column of a cellular automaton (without memory) whose
initial condition is eventually periodic in both directions.
\end{corollary}

\begin{proof}
\reftext{Corollary~\ref{without_memory}} provides a cellular automaton
on the alphabet $\mathbb F_q^{d+r+1}$ such that the sequence in
\xch{column}{Column}~$-2$, when projected onto first entries, is
$(u_n)_{n\geq0}$. We modify the cellular automaton to implement this
projection and produce a column which is the sequence $(u_n)$. Dilate
the existing spacetime diagram spatially by adding a new column between
every pair of consecutive existing columns. Adjust the local rule
correspondingly so that the retained (now ``even-indexed'') columns
emulate the original spacetime diagram. In each new (``odd-indexed'')
column, let the state of each cell be the first entry of the cell to
its left on the previous step. These two cases combine to form a local
rule on the alphabet $\mathbb F _q^{d+r+1} \cup\mathbb F_q$, since if
the value of a cell is in $\mathbb F_q$ then the rules ``knows'' to
perform the coding and otherwise to perform a linear rule on tuples.
Finally, remove the top row and use the second row as the initial
condition, since $u_0$ is the projection of an entry on the top row and
therefore is an entry on the second row.
\end{proof}

Note that the cellular automaton constructed in
\reftext{Corollary~\ref{without_projection}} is not linear; in
particular, the alphabet on which it is defined is no longer a group.

If $\varPhi$ and $\varPsi$ are two cellular automata with memory $d$ such that
\begin{eqnarray*}
\varPsi\Left{bigl}(\varPhi(R_0, R_1, \dots,
R_{d-1}), R_{d-1}, \dots, R_1\Right{bigr}) =
R_0
\end{eqnarray*}
for all $R_0, \dots, R_{d-1} \in\mathcal{A}^\mathbb Z$, we say that
$\varPhi$ is \emph{invertible}. The spacetime diagram of an invertible
cellular automaton can be evolved backward in time as well as forward,
just as the spacetime diagram of a cellular automaton whose local rule
is a bijective function of the leftmost or rightmost dependent cell can
be continued up the page~\cite{h,Rowland}.
\reftext{Fig.~\ref{inverted Rudin--Shapiro}} shows the spacetime
diagram of such an automaton.

\begin{corollary}\label{invertible}
If $(u_n)_{n \geq0}$ is a $p$-automatic sequence, then for some $r
\geq0$ the sequence $(u_n)_{n \geq r}$ occurs as a column of an
invertible cellular automaton with memory.
\end{corollary}

\begin{proof}
It suffices to arrange that $C_d(x)$ is a nonzero monomial, since
solving\break $\sum_{i=0}^d C_i(x) R_{n-i}(x) = 0$ for $R_{n-d}(x)$ then
gives a linear local rule for each entry on row $n-d$ in terms of
entries on rows $n-d+1, \dots, n$.

We may assume that $(u_n)_{n \geq0}$ has infinitely many zero terms,
since if not then some permutation of $\mathbb F_q$ results in a sequence
$(v_n)_{n \geq0}$ with infinitely many zero terms, and after
constructing a spacetime diagram containing $(v_n)_{n \geq r}$ we can
apply the inverse permutation to obtain a spacetime diagram containing
$(u_n)_{n \geq r}$.

As in the proof of \reftext{Proposition~\ref{technical_1}}, start with
the polynomial
\begin{eqnarray*}
P^*(t,x) = A^*_{0}(t) x + A^*_{1}(t) x^{p} +
\cdots+ A^*_{m}(t) x^{p^m} + B^*(t)
\end{eqnarray*}
where $A^*_{0}(0) \neq0$, whose existence is guaranteed by
\reftext{Proposition~\ref{prop_1}}.
However, instead of letting $r = r^* + 1$ as in the proof of
\reftext{Proposition~\ref{technical_1}}, we determine $r$ as follows.

Observe that the polynomial $P^*(t, u_{r^*} + t x)$ has the same form
as $P^*(t,x)$ but has the property that, for each $i$ such that $1 \leq
i \leq m$, the coefficient of $x^{p^i}$ is more highly divisible by $t$
than $A^*_i(t)$ is. Similarly, $P^*(t, u_{r^*} + t (u_{r^* + 1} + t
x))$ has the same form again but with coefficients that are even more
highly divisible by~$t$. Under this iterative substitution $x \mapsto
u_n + t x$ for $n = r^*,r^* + 1, \dots{}$, the exponent of $t$ grows
fastest in the coefficient of $x^{p^m}$, so there exists $r^{**} \geq
r^*$ such that the highest power of $t$ in the polynomial
\begin{equation*}
P^*(t, u_{r^*} + u_{r^* + 1} t + \cdots+
u_{r^{**}} t^{r^{**} - r^*} + t^{r^{**} - r^* + 1} x)
\end{equation*}
appears only in the coefficient of $x^{p^m}$.
Let $r \geq r^{**} + 1$ such that $u_r = 0$, which exists since
$(u_n)_{n \geq0}$ has infinitely many zeros.

Now resume the proof of \reftext{Proposition~\ref{technical_1}}:
The series $x = G(t) :=\sum_{n \geq1} u_{n+r} t^n$ is a zero
of the polynomial
\begin{eqnarray*}
P^*(t, u_{r^*} + u_{r^* + 1} t + \cdots+
u_{r - 1} t^{r - 1 - r^*} + 0 t^{r - r^*} + t^{r - r^*} x),
\end{eqnarray*}
and moreover the highest power of $t$ in this polynomial appears only
in the coefficient of $x^{p^m}$ (and not also in the coefficient of
$x^0$, as it would have if $u_r \neq0$). Therefore $C_d(x)$ is a
monomial. By dividing by an appropriate power of~$t$, the conclusions
of \reftext{Proposition~\ref{technical_1}}, and hence
\reftext{Theorem~\ref{cellular_automaton_construction}}, are
preserved.
\end{proof}

If the polynomials $A_{m}(t)$ and $B(t)$ are monomials, then we can
also think of time as moving to the left and right; we use this idea in
Section~\ref{factors} to prove \reftext{Corollary~\ref{embedded}}.

We conclude this section with some open questions suggested by the
previous results.
\begin{itemize}
\item
\reftext{Corollary~\ref{without_projection}} provides an upper bound of
$q^{d+r+1} + |\mathcal A|$ for the number of states in a cellular
automaton spacetime diagram containing a given $p$-automatic sequence
as a column, where the alphabet of the sequence is $\mathcal A \subset
\mathbb F_q$. Can this bound be improved?

\item
Each column in the spacetime diagram constructed in
\reftext{Corollary~\ref{invertible}} is bi-infinite. Does every
letter-to-letter projection of a bi-infinite fixed point of a
$p$-uniform substitution (see Section~\ref{factors}) occur as a
column of a bi-infinite spacetime diagram?

\item
Does there exist a $3$-automatic sequence $(u_n)_{n \geq0}$ on a
binary alphabet such that $(u_n)$ is not eventually periodic and
$(u_n)$ occurs as a column of a (nonlinear) $2$-state spacetime diagram?
\reftext{Theorem~\ref{automatic_sequence_construction}} rules out the
possibility of the rule being linear over $\mathbb F_2$ since a sequence
which is both $2$-automatic and $3$-automatic is eventually periodic by
Cobham's theorem.

\item
Which $k$-automatic sequences (if $k$ is not a prime power) occur as
columns of cellular automaton spacetime diagrams?
\end{itemize}

\section{Examples}\label{examples}

Provided that we have an algorithm for generating the required
polynomial in \reftext{Proposition~\ref{technical_1}}, the proof of
\reftext{Theorem~\ref{cellular_automaton_construction}} shows us how to
build the required cellular automaton.

An inspection of the proof of \reftext{Theorem~\ref{Christol}} yields
an algorithm for producing an explicit polynomial equation satisfied by
$F(t) = \sum_{n \geq0} u_n t^n$. For, given the $p$-DFAO that defines
$(u_{n})_{n\geq0}$, we can build the $q$-DFAO $(\mathcal S,
\varSigma_{q},\delta, s_{0}, \mathbb F_{q}, \omega)$ that defines
$(u_{n})_{n\geq0}$. Next, the $q$-kernel of $(u_{n})_{n\geq0}$ consists
of the $q$-automatic sequences built with the DFAO $(\mathcal S,
\varSigma_{q},\delta, s, \mathbb F_{q}, \omega)$, where the initial
state $s$ varies over~$\mathcal{S}$. If the $q$-kernel of
$(u_{n})_{n\geq0}$ contains $d$ elements, let them be generated by
initial states $s_{1}, \ldots, s_{d}$: thus each $s_{i}$ determines a
sequence $(u_{n}^{(i)})_{n\geq0}$ in the $q$-kernel of $(u_{n})_{n\geq
0}$. We can then\vspace{1.3pt} write each generating function $F_{i}(t)$ of
$(u_{n}^{(i)})_{n\geq0}$ as an explicit linear combination of the
functions $F_{1}(t^{q}),\ldots, F_{d}(t^{q})$. We repeat this procedure
$d$ times to explicitly express each $F_{j}(t^{q^i})$, for $i$ and $j$
satisfying $1\leq j\leq d$ and $0\leq i\leq d$, as a linear
combination, over the field $\mathbb F_{q}(x)$, of the elements in
$\{F_{1}(t^{q^{d+1}}), \ldots, F_{d}(t^{q^{d+1}})\}$. Now we have a
linear relationship between $F_{s_0}(t), \ldots, F_{s_0}(t^{q^{d}})$,
and this yields a polynomial that is almost of the form required by
\reftext{Proposition~\ref{prop_1}} --- the polynomial $P^{*}(t,x) $ may
need to be modified so that $A_{0}^{*}(t)$ is not divisible by~$t$.
Inspection of the proof of \reftext{Proposition~\ref{prop_1}} in
\cite{Fur} shows that this modification can be done mechanically. We
refer the interested reader to the proof of Theorems~6.6.2 and~12.2.5
in~\cite{ash}.

As examples, next we compute the cellular automaton rules and initial
conditions that generate three well-known automatic sequences as columns.

\begin{example}\label{Thue--Morse_computation}
First let us compute the cellular automaton for the Thue--Morse
sequence shown in \reftext{Fig.~\ref{Thue--Morse}}. Christol's theorem
gives $t x + (1+t) x^2 + (1+t^4) x^4 = 0$ satisfied by $x = F(t) =
\sum_{n \geq0} u_n t^n$.

Apply the proof of \reftext{Proposition~\ref{prop_1}} to this
polynomial. The coefficient of $x^1$ is already nonzero. However, since
it is divisible by $t$ we find an appropriate $r$ such that
replacing\footnote{Formally we are making the substitution $x =
\sum_{n=0}^{r-2} u_n t^n + t^{r-1} y$ but as we will be making
additional substitutions we prefer to be slightly sloppy than overly
complicated.} $x$ with $\sum_{n=0}^{r-2} u_n t^n + t^{r-1} x$ and
dividing by common powers of $t$ leaves the coefficient of $x^1$ not
divisible by~$t$. In this case $r = 2$ suffices, so we replace $x
\mapsto0 + t x$ and divide by $t^2$. Then $x = G^*(t) :=\sum_{n \geq0}
u_{n+1} t^n$ satisfies
\begin{equation*}
x + (1+t) x^2 + \Left{bigl}(t^2+t^6 \Right{bigr}) x^4 = 0.
\end{equation*}

Now apply \reftext{Proposition~\ref{technical_1}}. Replace $x$ with
$u_{r-1} + u_r t + t x = 1 + 1 t + t x$ so that $x = G(t) :=\sum_{n
\geq1} u_{n+2} t^n$ satisfies $P(t, G(t)) = 0$, where
\begin{eqnarray*}
P(t, x) = \Left{bigl}(t^2+t^9\Right{bigr}) + x +
\Left{bigl}(t+t^2\Right{bigr}) x^2 + \Left{bigl}(t^5+t^9
\Right{bigr}) x^4.
\end{eqnarray*}

Note that $P^{(0,1)}(t,x) = 1$. By \reftext{Proposition~\ref{prop_2}},
$u_{n+2}$ is the coefficient of $x^{-2}$ in $R_n(x)$ for all $n \geq1$,
where $R_n(x)$ is the coefficient of $t^n$ in the series
\begin{eqnarray*}
\frac{P^{(0,1)}(t,x)}{P(t,x)}
 &=& \frac{1}{x} \sum_{n \geq0}
\Left{biggl}(\frac{x - P(t,x)}{t x} \Right{biggr})^n t^n =
\frac{1}{x} + t + \Left{biggl}(\frac{1}{x^2} + 1 + x \Right{biggr})
t^2 + \cdots \\
&=& \sum_{n \geq0} R_n(x) t^n.
\end{eqnarray*}
By collecting the terms of $P(t,x)$ by common powers of $t$, we see
that $R_n(x)$ satisfies the recurrence
\begin{eqnarray*}
R_n(x) = x R_{n-1}(x) + \Left{biggl}(\frac{1}{x} +
x \Right{biggr}) R_{n-2}(x) + x^3 R_{n-5}(x) +
\Left{biggl}(\frac{1}{x} + x^3 \Right{biggr}) R_{n-9}(x)
\end{eqnarray*}
for all $n \geq10$. This recurrence determines a linear cellular
automaton rule $\varPhi$ with memory~$9$. Extend the memory to $d + r +
1 = 12$ without introducing dependence on the earliest $r + 1 = 3$
rows. Let $R_{-2}(x) = u_0 x^{-2} = 0$, $R_{-1}(x) = u_1 x^{-2} =
x^{-2}$, and $R_0(x) = u_2 x^{-2} = x^{-2}$. Then the sequence
$(u_n)_{n \geq0}$ occurs in \xch{column}{Column}~$-2$ of the spacetime
diagram of $\varPhi$ begun from initial conditions $R_{-2}, \dots ,
R_9$. Columns $-7$ through $12$ of rows $R_{-2}, \dots, R_{13}$ appear
below, with \xch{column}{Column} $-2$ highlighted.
\begin{center}
	\scalebox{.8}{\includegraphics{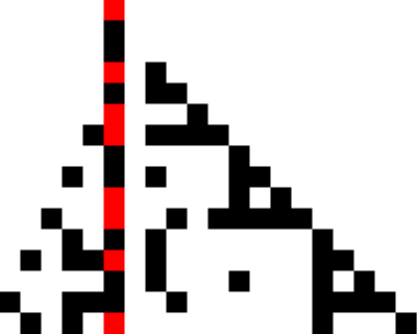}}
\end{center}

Note that the polynomial we computed from Christol's theorem is not the
minimal polynomial of $F(t)$, since $x = F(t)$ is also a zero of $t +
(1 + t^2) x + (1 + t + t^2 + t^3) x^2$. Using this polynomial instead
produces a different polynomial
\begin{eqnarray*}
P(t,x) = \Left{bigl}(t+t^3+t^4\Right{bigr}) +
\Left{bigl}(1+t^2\Right{bigr}) x + \Left{bigl}(t+t^2+t^3+t^4
\Right{bigr}) x^2
\end{eqnarray*}
and hence a different cellular automaton.
In fact $d = 4$ and $r = 1$, so the memory is lowered to $6$.
\end{example}

\begin{example}
In \reftext{Example~\ref{Thue--Morse_computation}} the coefficient of
$x^1$ in the polynomial obtained from Christol's theorem was nonzero,
and this saved some work. Therefore let us work out example where none
of the steps in the algorithm are trivial. Additionally, we go through
the construction of the polynomial from Christol's theorem. The
Rudin--Shapiro sequence is the $2$-automatic sequence $(u_n)_{n \geq0}
= 0, 0, 0, 1, 0, 0, 1, 0, \dots$ where $u_n = 0$ if the number of
(possibly overlapping) occurrences of $11$ in the binary representation
of $n$ is even and $u_n = 1$ otherwise.

First we apply \reftext{Theorem~\ref{Christol}} as follows to construct
a polynomial equation satisfied by $F(t) = \sum_{n \geq0} u_n t^n = t^3
+ t^6 + t^{11} + t^{12} + \cdots{}$. There are $|\mathcal S| = 4$
sequences in the $2$-kernel of $(u_n)$: $(u_n) = (u_{2n}) =
(u_{4n+1})$, $(u_{2n+1}) = (u_{8n+7})$, $(u_{4n+3}) = (u_{16n+11})$,
and $(u_{8n+3}) = (u_{16n+3})$. Each of the generating functions of
these four sequences can be broken up into its even- and odd-index
terms and rewritten in terms of the original four generating functions:
\begin{align*}
	F(t) = F_1(t) &\colonequal \sum_{n \geq 0} u_n t^n = F_1(t^2) + t \, F_2(t^2) \\
	F_2(t) &\colonequal \sum_{n \geq 0} u_{2n+1} t^n = F_1(t^2) + t \, F_3(t^2) \\
	F_3(t) &\colonequal \sum_{n \geq 0} u_{4n+3} t^n = F_4(t^2) + t \, F_2(t^2) \\
	F_4(t) &\colonequal \sum_{n \geq 0} u_{8n+3} t^n= F_4(t^2) + t \, F_3(t^2).
\end{align*}
Using these four equations iteratively, we may write each $F(t^{2^i})$
for $0 \leq i \leq|\mathcal S|$ as a linear combination of
$F_j(t^{2^{|\mathcal{S}|+1}})$.
The result of doing so is
\tiny
\[
	\begin{pmatrix}
		a_{11} & a_{12} & a_{13} & a_{14} \\
		1 + t^2 + t^4 + t^8 + t^{10} + t^{14} & t^{16} + t^{18} + t^{20} + t^{28} & t^{22} + t^{24} + t^{26} + t^{30} & t^6 + t^{12} \\
		1 + t^4 + t^8 & t^{16} + t^{20} + t^{28} & t^{24} & t^{12} \\
		1 + t^8 & t^{16} & t^{24} & 0 \\
		1 & t^{16} & 0 & 0
	\end{pmatrix}
	\begin{pmatrix}
		F_1(t^{32}) \\
		F_2(t^{32}) \\
		F_3(t^{32}) \\
		F_4(t^{32})
	\end{pmatrix}
	=
	\begin{pmatrix}
		F(t) \\
		F(t^2) \\
		F(t^4) \\
		F(t^8) \\
		F(t^{16})
	\end{pmatrix}
\]
\normalsize
where
\begin{eqnarray*}
a_{11} &=& 1 + t + t^2 + t^4 +
t^5 + t^7 + t^8 + t^9 +
t^{10} + t^{14}
\eqncr
\noalign{\vspace{-1pt}}
a_{12} &=& t^{16}
+ t^{17} + t^{18} + t^{20} + t^{21} +
t^{23} + t^{27} + t^{28} + t^{29} +
t^{31}
\eqncr
\noalign{\vspace{-1pt}}
a_{13} &=& t^{19} + t^{22}
+ t^{24} + t^{25} + t^{26} + t^{30}
\eqncr
\noalign{\vspace{-1pt}}
a_{14} &=& t^3 + t^6 + t^{11}
+ t^{12} + t^{13} + t^{15}.
\end{eqnarray*}\vspace{-1pt}%
\noindent
We have a system of equations in $9$ variables (the series
$F_j(t^{32})$ for $1 \leq j \leq4$ and $F(t^{2^i})$ for $0 \leq i \leq
4$), and performing Gaussian elimination on the corresponding $5 \times
9$ matrix gives, in the bottom row, the coefficients of a polynomial
equation satisfied by $x = F(t)$, namely\vspace{-1pt}
\begin{equation*}
t^6 x^2 + \Left{bigl}(1 + t^6\Right{bigr})
x^4 + \Left{bigl}(1 + t^4 + t^8 +
t^{12}\Right{bigr}) x^8 = 0.
\end{equation*}

Next we apply \reftext{Proposition~\ref{prop_1}}. Since $x^1$ does not
appear in the polynomial equation we have found, extract terms whose
power of $t$ is even (which in this case is all terms) and raise both
sides to the power~$1/2$. The resulting equation
\begin{eqnarray*}
t^3 x + \Left{bigl}(1 + t^3\Right{bigr})
x^2 + \Left{bigl}(1 + t^2 + t^4 +
t^6\Right{bigr}) x^4 = 0
\end{eqnarray*}
has a nonzero coefficient of $x^1$, as desired.
To obtain a coefficient of $x^1$ that is not divisible by~$t$, let $r =
4$, replace $x \mapsto0 + 0 t + 0 t^2 + t^3 x$, and divide by $t^6$.
Then $x = G^*(t) :=\sum_{n \geq0} u_{n+3} t^n$ satisfies
\begin{equation*}
x + \Left{bigl}(1 + t^3\Right{bigr}) x^2 +
\Left{bigl}(t^6 + t^8 + t^{10} +
t^{12}\Right{bigr}) x^4 = 0.
\end{equation*}

Finally, we apply \reftext{Proposition~\ref{technical_1}}.
Replace $x$ with $1 + 0 t + t x$ so that $x = G(t) :=\sum_{n
\geq1} u_{n+4} t^n$ satisfies $P(t, G(t)) = 0$, where
\begin{eqnarray*}
P(t, x) = \Left{bigl}(t^2 + t^5 + t^7 + t^9 + t^{11}\Right{bigr}) + x +
\Left{bigl}(t + t^4\Right{bigr}) x^2 + \Left{bigl}(t^9 + t^{11} +
t^{13} + t^{15}\Right{bigr}) x^4.
\end{eqnarray*}
By \reftext{Proposition~\ref{prop_2}}, $u_{n+4}$ is the coefficient of
$x^{-2}$ in $R_n(x)$ for all $n \geq1$, where
\begin{eqnarray*}
\sum_{n \geq0} R_n(x) t^n =
\frac{1}{x} + t + \Left{biggl}(\frac{1}{x^2} + x \Right{biggr})
t^2 + x^2 t^3 + \Left{biggl}(\frac{1}{x^3} + x^3 \Right{biggr}) t^4 +
x^4 t^5 + \cdots.
\end{eqnarray*}
Moreover, $R_n(x)$ satisfies the recurrence
\begin{eqnarray*}
R_n(x) &=& x R_{n-1}(x) + \frac{1}{x}
R_{n-2}(x) + x R_{n-4}(x) + \frac{1}{x}
R_{n-5}(x) + \frac{1}{x} R_{n-7}(x)
\\
&&{} + \Left{biggl}(\frac{1}{x} + x^3 \Right{biggr})
R_{n-9}(x) + \Left{biggl}(\frac{1}{x} + x^3
\Right{biggr}) R_{n-11}(x) + x^3 R_{n-13}(x) +
x^3 R_{n-15}(x)
\end{eqnarray*}
for all $n \geq16$.
Therefore the cellular automaton rule $\varPhi$ has memory $15$, which we
increase to $20$ to reinstate the initial rows.
The first $256$ rows of the resulting spacetime diagram appear in
\reftext{Fig.~\ref{Rudin--Shapiro}}.

\begin{figure}
	\begin{center}
		\scalebox{.8}{\includegraphics{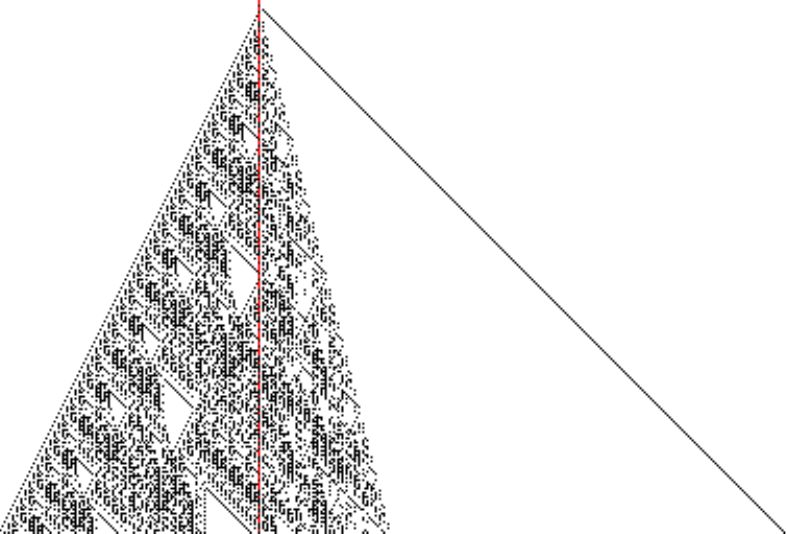}}
		\caption{Spacetime diagram of a cellular automaton with memory $20$ containing the Rudin--Shapiro sequence.}
		\label{Rudin--Shapiro}
	\end{center}
\end{figure}

\begin{figure}
	\begin{center}
		\scalebox{.8}{\includegraphics{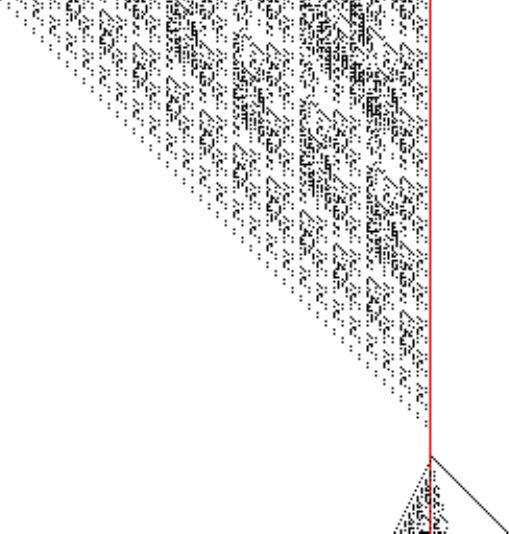}}
		\caption{Spacetime diagram showing the beginning of the infinite history of an invertible cellular automaton containing the Rudin--Shapiro sequence.}
		\label{inverted Rudin--Shapiro}
	\end{center}
\end{figure}

Note that since $u_r$ happened to be $0$, the highest power of $t$ in
$P(t,x)$ appears only in the coefficient of $x^4$, and therefore the
cellular automaton is invertible in accordance with
\reftext{Corollary~\ref{invertible}}.
\reftext{Fig.~\ref{inverted Rudin--Shapiro}} shows the spacetime
diagram for rows $R_{-215}$ through $R_{40}$.\looseness=1
\end{example}

\begin{example}
The Baum--Sweet sequence is the $2$-automatic sequence $(u_n)_{n \geq
0} = 1, 1, 0, 1, 1, 0, 0, 1, \dots$ where $u_n = 0$ if the binary
representation of $n$ contains a block of $0$s of odd length and $u_n =
1$ if not. (Note we consider the binary representation of $0$ to be the
empty word.) Christol's theorem gives $t^2 x + (1+t^3+t^4) x^2 + t^6
x^4 + (1+\nobreak t^4) x^8 = 0$ satisfied by $x = \sum_{n \geq0} u_n t^n$. The
output of \reftext{Proposition~\ref{technical_1}} is the
polynomial\vspace{-1pt}
\begin{eqnarray*}
P(t,x) &=& \Left{bigl}(t+t^3+t^4+t^7+t^{13}+t^{19}+t^{23}
\Right{bigr}) + x + \Left{bigl}(t+t^4+t^5\Right{bigr})
x^2 + t^{13} x^4 \\
&&{} + \Left{bigl}(t^{19}+t^{23}\Right{bigr}) x^8.
\end{eqnarray*}\vspace{-1pt}%
\noindent
Therefore we have a cellular automaton with memory $d + r + 1 = 23 + 3
+ 1 = 27$. The first $192$ rows appear in
\reftext{Fig.~\ref{Baum--Sweet}}.
\end{example}

\section{Substitution dynamical systems as factors of cellular
automata}\label{factors}\vspace{-1pt}

\begin{figure}
	\begin{center}
		\scalebox{.8}{\includegraphics{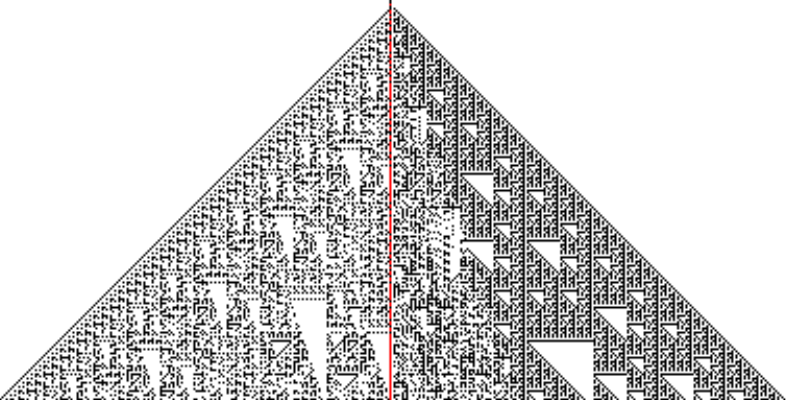}}
		\caption{Spacetime diagram of a cellular automaton with memory $27$ containing the Baum--Sweet sequence.}
		\label{Baum--Sweet}
	\end{center}
\end{figure}

In this section we apply \reftext{Theorem~\ref{characterization}} to conclude
that certain dynamical systems arise as factors of cellular automata.
First we define some terms.

\begin{definition}
Let $(X,S)$ and $(Y,T)$ be two dynamical systems.\vspace{-1pt}
\begin{enumerate}
\item If $X$ is a closed subset of $Y$ and $T(X)\subset X$, then
we say that $(X, T)$ is a \emph{subsystem of $(Y,T)$}.
\item If there exists a homeomorphism $\varPsi: Y\rightarrow X$ with
$\varPsi\circ T = S\circ\varPsi$, we say
the dynamical systems $(X, S)$ and $(Y,T)$ are \emph{topologically conjugate}.
\item If $(X,S)$ is conjugate to a subsystem of $(Y,T)$, then we
say that $(Y,T)$ \emph{embeds} $(X,S)$.
\item If there exists a continuous surjective mapping $\varPsi:Y
\rightarrow X$ such that $S\circ\varPhi= \varPhi\circ T$, we say $(X,S)$ is
a (topological) \emph{factor} of $(Y,T)$.\vspace{-1pt}
\end{enumerate}
\end{definition}

\begin{definition}
If $\mathbf{u}\in{\mathcal A}^{\mathbb N}$, define $X_{\mathbf{u}}
:= \overline{\{\sigma^{n}(\mathbf{u}): n\in\mathbb N\}}$.
The dynamical system $(X_{\mathbf{u}}, \sigma)$ is called the (one-sided)
subshift associated with $\mathbf{u}$.\vspace{-1pt}
\end{definition}

\begin{factortheorem*}
Let $\mathbf{u}$ be $p$-automatic. Then $(X_{\mathbf{u}},\sigma)$ is a
factor of a subsystem of some linear cellular automaton $((\mathbb
F_{q}^{d})^{\mathbb Z}, \varPhi)$.\vspace{-1pt}
\end{factortheorem*}

\begin{proof}
By \reftext{Corollary~\ref{without_memory}}, $\mathbf{u}$ is the image,
under a letter-to-letter projection, of a sequence $\mathbf{v}$ which
appears as a column in the \xch{spacetime}{space-time} diagram, with
initial condition $R_0$, of a linear cellular automaton $\varPhi$. We
shall show that $(X_{\mathbf {v}}, \sigma )$ is a factor of a subsystem
of $((\mathbb F_{q}^{d+r+1})^{\mathbb Z}, \varPhi )$; the fact that
$\mathbf{u}$ is a letter-to-letter projection of $\mathbf{v}$ implies
that $(X_{\mathbf{u}}, \sigma)$ is a factor of $(X_{\mathbf {v}},
\sigma)$.

Define the map $\varPsi:\{{\varPhi}^{n}({R_0})\}_{n\geq0}\rightarrow\{
\sigma ^{n}({\mathbf{v}})\}_{n\geq0}$ as $\varPsi({\varPhi
}^{n}({R_0})):=\sigma ^{n}({\mathbf{v}})$. Since $\mathbf{v}$ is a
column of the \xch{spacetime}{space-time} diagram of $\varPhi$ with
initial condition $R_0$, and $\varPhi$ is defined by a local rule, then
it is straightforward to see that $\varPsi$ is uniformly continuous. We
claim that the map $\varPsi:\{{\varPhi}^{n}({R_0})\}\rightarrow\{\sigma
^{n}({\mathbf{v}})\}$ extends to a continuous surjection
$\varPsi:\overline{\{{\varPhi}^{n}({R_0})\}}\rightarrow X_{\mathbf{v}}$
satisfying $\varPsi\circ\varPhi= \sigma\circ\varPsi$. The proof is
standard but we include it. Fix
$R\in\overline{\{{\varPhi}^{n}({R_0})\}}$ and suppose that
$\varPhi^{n_k}(R_0)\rightarrow R$. We will show that the set
$\{\sigma^{n_k}(\mathbf{v})\}$ has a unique limit point~$\mathbf{y}$,
which is independent of the sequence $(n_k)$. Given~$\epsilon$, the
uniform continuity of $\varPsi$ on the $\varPhi$-orbit of $R_0$ implies
that we can find a $\delta$ such that
$d(\sigma^{n}(\mathbf{v}),\sigma^{m}(\mathbf{v}))<\epsilon/3$ whenever
$d(\varPhi^{n}(R_0), \varPhi^{m}(R_0))<\delta$ (where $d$ is the metric
generated by the topology on the relevant Cantor space). Thus if
$\sigma^{n_{k_l}}(\mathbf{v})\rightarrow\mathbf{y}$ and $\sigma
^{n_{k'_l}}(\mathbf{v} )\rightarrow\mathbf{y}'$, then there is an $L$
such that if $l \geq L$ then $d(\sigma ^{n_{k_l}}(\mathbf{v}),
\sigma^{n_{k'_l}}(\mathbf{v}))<\epsilon/3$. If $L $ is also large
enough so that $\sigma^{n_{k_l}}(\mathbf{v})$, $\sigma
^{n_{k'_l}}(\mathbf{v})$ are $\epsilon/3$-close to $\mathbf{y}$,
$\mathbf{y}'$ respectively, then $d(\mathbf{y},\mathbf{y} ')<\epsilon$.
Hence $\mathbf{y}=\mathbf{y}'$.

Now suppose that $\varPhi^{m_k}(R_0)\rightarrow R$. The proximity of
$\varPhi^{n_k}(R_0)$ and $\varPhi^{m_k}(R_0)$, for large~$k$, implies
the proximity of $\sigma^{n_k}(\mathbf{v})$ and
$\sigma^{m_k}(\mathbf{v})$ for large~$k$, which implies that the limit
point of each of the sets $\{\sigma^{n_k}(\mathbf{v})\}$ and
$\{\sigma^{m_k}(\mathbf{v})\}$ is the same; let this limit point be
$\mathbf{y}$. We can now define $\varPsi(R)=\mathbf{y}$. To see that
$\varPsi$ is continuous, note that if $R$ and $R'$ are close, and
$\varPhi^{n_k}(R_0)\rightarrow R$, $\varPhi^{n'_{k}}(R_0)\rightarrow
R'$, then for large $k$ $\varPhi^{n_k}(R_0)$ and $\varPhi^{n_k'}(R_0)$
are close, which implies that $\mathbf{y}$ and $\mathbf{y}'$ are close.
To see that $\varPsi$ is surjective: if
$\sigma^{n_k}(\mathbf{v})\rightarrow \mathbf{y}$, then let $R$ be a
limit point of $\varPhi^{n_k}(R_0)$: then $\varPsi(R)=\mathbf{y}$.
Finally if $\varPhi^{n_k}(R_0)\rightarrow R$ and
$\sigma^{n_k}(\mathbf{u}) \rightarrow\mathbf{y}=\varPsi(R)$ then\vspace{-1pt}
\begin{eqnarray*}
\varPsi\circ\varPhi(R) = \varPsi\circ\varPhi\Left{bigl}(\lim
\varPhi^{n_k}(R_0)\Right{bigr})= \lim_{k}
\sigma^{n_{k}+1}(\mathbf{u}) = \sigma\lim_{k}
\sigma^{n_k}(\mathbf {u}) = \sigma\Left{bigl}(\varPsi (R)\Right{bigr}).
\qedhere
\end{eqnarray*}
\xmlstring{\qed}\goodbreak
\end{proof}

We now define a class of well-studied subshifts, and state Cobham's
theorem, which tells us that these subshifts arise from $p$-automatic sequences.

\begin{definition}
Let $\mathcal S$ be a finite alphabet. A \emph{substitution} (or
\emph{morphism}) is a map $\tau: {\mathcal S} \rightarrow{\mathcal
S}^{+}$. The map $\tau$ extends to a map $\tau: {\mathcal
S}^{+}\cup\mathcal S^{\mathbb N}\rightarrow{\mathcal S}^{+}\cup
\mathcal S^{\mathbb N}$ by concatenation: if ${\mathbf{a}}= a_{1}\cdots
a_{k}\cdots$, then $\tau({\mathbf{a}}) :=\tau(a_{1})
\cdots\tau(a_{k})\cdots{}$.
\end{definition}

\begin{definition}
Let $\tau$ be a substitution on $\mathcal S$. 
If $|\tau(a)|=k$ for each $a \in \mathcal S$, we say that 
$\tau$ is a \emph{length-$k$} substitution (or a \emph{$k$-uniform morphism}).
\end{definition}

\begin{definition}
A \emph{fixed point} of $\tau$ is a sequence $\mathbf{v}= (v_n)_{n
\geq0} \in{\mathcal S}^{\mathbb N}$ such that $\tau({\mathbf{v}}) =
{\mathbf{v}}$.
\end{definition}

Cobham's theorem \cite{C} gives us the relationship between
$k$-automatic sequences and fixed points of length-$k$ substitutions:

\begin{theorem}\label{Cobham}
A sequence is $k$-automatic if and only if it is the image, under a
letter-to-letter projection, of a fixed point of a length-$k$
substitution.
\end{theorem}

Dynamicists have extensively studied substitution subshifts ---
references detailing some of this work include \cite{pf} and \cite{qu}.
Combining Cobham's theorem with \reftext{Theorem~\ref{factor}}, we
obtain the following.

\begin{corollary}
Let $\mathbf{v}$ be a fixed point of a length-$p$
substitution. Then $(X_{\mathbf{v}},\sigma)$ is a factor of a subsystem
of some linear cellular automaton $((\mathbb F_{q}^{d})^{\mathbb Z},
\varPhi)$.
\end{corollary}

It would be interesting to know whether the factor mapping in
\reftext{Theorem~\ref{factor}} is, in some or all cases, an embedding.
This is in principle possible: in \cite{py6}, substitution systems are
embedded in subsystems of cellular automata; however the cellular
automata are tailored for the specific substitution and have no nice
algebraic or combinatorial structure. We end with an extra condition on
the polynomial given by \reftext{Proposition~\ref{technical_1}} which
would give an embedding of the substitution subshift into a cellular
automaton, and leave as an open question whether such a polynomial
always exists.

\begin{corollary}\label{embedded}
Suppose, using the notation of \reftext{Proposition~\ref{technical_1}}
and \reftext{Theorem~\ref{cellular_automaton_construction}}, that both
$A_{m}(t) = \alpha t^{d}$ and $B(t) = \beta t^{d}$ are monomials of
degree~$d$. Then $(X_{\mathbf{u}},\sigma)$ is the letter-to-letter
projection of a subshift that can be embedded in a linear cellular
automaton.
\end{corollary}

\begin{proof}
Recall that $R_{n}(x)= -\sum_{i=1}^{d}\frac{C_{i}(x)}{C_{0}(x)}
R_{n-i}(x)$. If $A_{m}(t)$ is a monomial, then only one of the
polynomials $C_{i}(x)$, say $i_{R}$, has the $x^{p^m}$ term. This means
that if the left radius of $\varPhi$ is $p^{m}-1$, and $\varPhi=
\sum_{i=1}^{d}\varPhi_{i}$ where $\varPhi_{i}$ is the cellular
automaton defined by $\frac{C_{i}(x)}{C_{0}(x)}$, then other than
$\varPhi_{d}$, all $\varPhi_{i}$'s have radius strictly less than
$p^{m}-1$. Rotating our original spacetime diagram $\mathcal
S_{\varPhi}$ counter-clockwise by 90 degrees, we see a new spacetime
diagram for another cellular automaton with memory $l+r+1$. Similarly,
only one of the cellular automata $\varPhi_{d}$ will have right radius
$l=1$, so that rotating $\mathcal S_{\varPhi}$ by 90 degrees clockwise,
we see another spacetime diagram for another cellular automaton with
memory $l+r+1$. This means that if in $\mathcal S_{\varPhi}$, the
central $l+r+1$ columns $C_{-l},\ldots, C_{r}$ have the same entries in
a large enough block of length $L$ starting at locations $m_{1}, m_{2}$
respectively, then the entries in two rows $R_{m_{1}}, R_{m_2}$ will
agree in a large central block. Thus if in the proof of
\reftext{Theorem~\ref{factor}}, we consider $X_{\mathbf{w}}$ where
$\mathbf{w}\in\mathcal S^{l+r+1}$ is the infinite word defined by the
columns $C_{-l},\ldots, C_{r}$, then the map $\varPsi$ defined in
\reftext{Theorem~\ref{factor}} is a topological conjugacy between
$(\overline{\{{\varPhi}^{n}({R_0})\}}, \varPhi)$ and $(X_{\mathbf {w}},
\sigma)$. Projecting $X_{\mathbf{w}}$ to the appropriate column
containing~$\mathbf{u}$, the result follows.
\end{proof}

\section*{Acknowledgments}

Most of this work was done during the second author's visit to McGill
University; we are thankful to this institution for its hospitality
and support. We would like also to thank M.~Pivato for useful
discussions, and J.-P.~Allouche for pointing out to us the results of
\cite{ahpps}.

\end{document}